\documentclass{amsart}[12pt]
\usepackage{amsmath, amsthm, amscd, amsfonts,}
\usepackage{graphicx,float}
\usepackage{comment}

\usepackage{amssymb}

\usepackage{pb-diagram}
\usepackage{todonotes}

\newcommand{\PP}{{\mathbb{P}}}

\newcommand{\url}[1]{{\tt #1}}

\DeclareMathOperator{\len}{lh}

\DeclareMathOperator{\dom}{dom}
\DeclareMathOperator{\cf}{cf}
\DeclareMathOperator{\Col}{Col}
\DeclareMathOperator{\Add}{Add}

\DeclareMathOperator{\supp}{supp}

\def\MPB{{\mathbb{P}}}
\def\MQB{{\mathbb{Q}}}
\def\MRB{{\mathbb{R}}}

\def\k{\kappa}

\def\l{\lambda}

\def\a{\alpha}

\def\b{\beta}

\newtheorem{theorem}{Theorem}[section]
\newtheorem{lemma}[theorem]{Lemma}

\newtheorem{definition}[theorem]{Definition}
\newtheorem{open Question}[theorem]{Open Question}

\newtheorem{remark}[theorem]{Remark}
\newtheorem{notation}[theorem]{Notation}
\newtheorem{claim}[theorem]{Claim}

\numberwithin{equation}{section}

\def\MPB{{\mathbb{P}}}
\def\MQB{{\mathbb{Q}}}
\def\MRB{{\mathbb{R}}}

\def\k{\kappa}

\def\l{\lambda}

\def\a{\alpha}

\def\b{\beta}

\def\l{\lambda}

\def\rmark{\mbox{$\rm\bf\rule{0.06em}{1.45ex}\kern-0.05em R$}}
\def\pmark{\mbox{$\rm\bf\rule{0.06em}{1.45ex}\kern-0.05em P$}}
\def\nmark{\mbox{$\rm\bf\rule{0.06em}{1.45ex}\kern-0.05em N$}}
\def\vdash{\mbox{$\rm\| \kern-0.13em -$}}
\newcommand{\lusim}[1]{\smash{\underset{\raisebox{1.2pt}[0cm][0cm]{$\sim$}}
{{#1}}}}

\newcommand{\ZFC}{\hbox{ZFC}}

\newcommand{\GCH}{\hbox{GCH}}
\newcommand{\CH}{\hbox{CH}}

\newcommand{\Lim}{\hbox{Lim }}

\usepackage{pb-diagram}

\def\l{\lambda}

\def\rmark{\mbox{$\rm\bf\rule{0.06em}{1.45ex}\kern-0.05em R$}}
\def\pmark{\mbox{$\rm\bf\rule{0.06em}{1.45ex}\kern-0.05em P$}}
\def\nmark{\mbox{$\rm\bf\rule{0.06em}{1.45ex}\kern-0.05em N$}}
\def\vdash{\mbox{$\rm\| \kern-0.13em -$}}

\begin{document}

\title[Specializing trees and answer to a question of Williams]{Specializing trees and answer to a question of Williams}

\author[ M. Golshani and S. Shelah]{Mohammad Golshani and Saharon Shelah}

\thanks{ The first author's research has been supported by a grant from IPM (No. 96030417). He also thanks
Camilo Arosemena and Yair Hayut for   the helpful comments.}
\thanks{ The second
author's research has been partially supported by European Research Council grant 338821. Publication 1120 on Shelah's list.}
\thanks{The authors thank Ashutosh Kumar  for  some useful comments and corrections. They also thank the referee of the paper for some useful comments and
corrections which improved the presentation of the paper.}
\maketitle

\begin{abstract}
We show that if $\cf(2^{\aleph_0})=\aleph_1,$ then any non-trivial $\aleph_1$-closed forcing
notion of size $\leq 2^{\aleph_0}$ is forcing equivalent to $\Add(\aleph_1, 1),$
the Cohen forcing for adding a new Cohen subset of $\omega_1.$ We also produce, relative to the existence of suitable large cardinals, a model of $\ZFC$ in which $2^{\aleph_0}=\aleph_2$ and all $\aleph_1$-closed forcing
notion of size $\leq 2^{\aleph_0}$ collapse $\aleph_2,$ and hence are forcing equivalent to $\Add(\aleph_1, 1).$ These results answer a
question of Scott Williams from 1978. We also extend a result of Todorcevic and Foreman-Magidor-Shelah by showing that  it is consistent that every partial order which adds a new subset of $\aleph_2,$ collapses $\aleph_2$ or $\aleph_3.$
\end{abstract}
\maketitle

\section{introduction}
For an infinite cardinal $\kappa,$ let $\Add(\kappa, 1)$ denote the Cohen forcing for adding a new Cohen subset of $\kappa$; thus conditions in $\Add(\kappa, 1)$ are partial functions
$p: \kappa \to \{0, 1\}$ of size less than $\kappa$, ordered by reverse inclusion. The forcing is  $\cf(\kappa)$-closed and satisfies  $(2^{<\kappa})^+$-c.c., in particular, if $\kappa$ is regular and $2^{<\kappa}=\kappa,$ then it preserves all cardinals.

It is well-known that if the continuum hypothesis, $\CH$, holds, then any $\aleph_1$-closed forcing notion of size continuum is forcing equivalent to
$\Add(\aleph_1, 1).$ In \cite{williams0} (see also \cite{williams1}), Scott Williams asked if the converse is also true, i.e., if $\CH$
follows from the assumption ``any $\aleph_1$-closed forcing notion of size continuum is forcing equivalent to the Cohen forcing
$\Add(\aleph_1, 1)$''.   We will show that  $\cf(2^{\aleph_0})=\aleph_1$ is sufficient to conclude that all $\aleph_1$-closed forcing notions of size continuum are forcing equivalent to
$\Add(\aleph_1, 1).$ Since  $\cf(2^{\aleph_0})=\aleph_1$ is consistent with $\neg \CH$, this gives a negative answer to Williams question.
\begin{theorem}
\label{main theorem1}
Assume $\cf(2^{\aleph_0})=\aleph_1.$ Then any non-trivial $\aleph_1$-closed forcing notion of size $\leq 2^{\aleph_0}$ is forcing equivalent to $\Add(\aleph_1, 1).$
\end{theorem}
\begin{remark}
$(1)$ We can replace $\aleph_1, 2^{\aleph_0}$ by $\kappa=\mu^+, 2^\mu$ resp.,  with $\cf(2^\mu)=\kappa;$ or by $\kappa, 2^{\mu}$ resp.,  if $\kappa$
is weakly inaccessible, $\mu<\kappa$, $2^\mu=2^{<\kappa}$ and  $\cf(2^\mu)=\kappa$.

$(2)$ If $2^{\aleph_0}=2^{\aleph_1},$ then $\Add(\aleph_2, 1)$ is $\aleph_1$-closed of size continuum, but it is not forcing equivalent to $\Add(\aleph_1,1)$.
\end{remark}

On the other hand, it is not difficult to prove the consistency of ``$2^{\aleph_0}=\aleph_2$  and there exists a non-trivial $\aleph_1$-closed (but not $\aleph_2$-closed) forcing notion of size $\aleph_2$ which preserves all cardinals'' (see \cite{g-h}).
So it is natural to ask if we can have the same result as in Theorem \ref{main theorem1} with
 $2^{\aleph_0}$ being regular. We show that this is indeed the case, if we assume the existence of  large cardinals.

Recall that an uncountable cardinal $\kappa$ is  \emph{supercompact}  if for every cardinal $\l > \k$ there exists a non-trivial elementary embedding
 $j: V \rightarrow M$ with critical point $\k$ such that $j(\k)> \l$ and $^{\l}M \subseteq M$. It is \emph{$2$-Mahlo} if
 $\{\mu < \k: \mu $ is a Mahlo cardinal$\}$ is stationary in $\k.$
 \begin{theorem}
 \label{main theorem2}
 Assume $\kappa$ is a supercompact cardinal and $\lambda>\kappa$ is a $2$-Mahlo cardinal. Then there is a generic extension of the universe in which the following hold:

 $(a)$ $2^{\aleph_0}=\kappa=\aleph_2,$

 $(b)$ $2^{\aleph_1}=\lambda=\aleph_3,$

 $(c)$ Any $\aleph_1$-closed forcing notion of size $\leq \aleph_2$ collapses $\aleph_2$ into $\aleph_1,$
 in particular it is forcing equivalent to $\Add(\aleph_1, 1).$
 \end{theorem}
 Following \cite{cox-krueger2}, let \emph{Todorcevic's maximality principle} be the assertion: ``every partial order which adds a fresh subset of $\aleph_1,$ collapses $\aleph_1$ or $\aleph_2$'', where by a fresh subset of a cardinal $\kappa$ we mean a subset of $\kappa$ which is not in the ground model but all of its proper initial segments are in
the ground model.

 In \cite{todorcevic},  Todorcevic showed that if $2^{\aleph_0}=\aleph_2$ and every $\aleph_1$-tree of size $\aleph_1$ is special, then
 Todorcevic's maximality principle holds.

 By results of Baumgartner \cite{baumgartner} and Todorcevic \cite{todorcevic2}, ``$2^{\aleph_0}=\aleph_2+$ every $\aleph_1$-tree of size $\aleph_1$ is special'' is consistent, and hence  Todorcevic's maximality principle is consistent as well. On the other hand, Foreman-Magidor-Shelah \cite{foreman-magidor-shelah} showed that $\text{PFA}$ implies the same conclusion.
In \cite{viale}, Viale and Weiss introduced the principle $\text{GMP}$ (guessing model principle)  and showed that it follows from $\text{PFA}$.
 Cox and Krueger  \cite{cox-krueger1}, introduced the stronger principle $\text{IGMP}$ (indestructible guessing model principle) and showed that
$\text{PFA}$ implies $\text{IGMP}$ which in turn implies Todorcevic's maximality principle. On the other hand, in \cite{cox-krueger2}, they showed that Todorcevic's maximality principle does not follow
from $\text{GMP}.$

We   extends the above result of Todorcevic to higher cardinals, and prove the following theorem.
\begin{theorem}
 \label{main theorem3}
Assume $\kappa$ is a supercompact cardinal and $\lambda>\kappa$ is a $2$-Mahlo cardinal.
Then there is a generic extension of the universe in which the following hold:

 $(a)$ $2^{\aleph_0}=\aleph_1,$

 $(b)$ $\kappa=\aleph_2,$

 $(c)$ $2^{\aleph_1}=\lambda=\aleph_3,$

 $(d)$ Every partial order which adds a fresh subset of $\aleph_2,$ collapses $\aleph_2$ or $\aleph_3.$
\end{theorem}
\begin{remark}
In theorems \ref{main theorem2} and \ref{main theorem3}, we can replace the cardinals $\aleph_0, \aleph_1$ and $\aleph_2$ by the cardinals $\eta, \eta^+$ and $\eta^{++}$ respectively, where $\eta$ is a regular cardinal less than $\kappa.$
\end{remark}
The above result is  related to Foreman's maximality principle \cite{foreman-magidor-shelah86}, which asserts that any non-trivial forcing notion either adds a new real
or collapses some cardinals.

The paper is organized as follows. In section 2 we prove Theorem \ref{main theorem1}. Sections 3 and 4 are devoted to some preliminary results
which are then used in section 5 for the proof of Theorem \ref{main theorem2}. In the last  section 6, we rove Theorem \ref{main theorem3}.

To avoid trivialities, by a forcing notion we always mean a non-trivial separative forcing notion. We  use $\simeq$ for the equivalence of forcing notions, so
\[
\MPB \simeq \MQB ~~\Leftrightarrow ~~ RO(\MPB) \text{~is isomorphic to~} RO(\MQB),
\]
where $RO(\MPB)$ denotes the Boolean completion of $\MPB.$ Also $\MPB \lessdot \MQB$ means that $\MPB$ is a regular sub-forcing of $\MQB.$
\section{A negative answer to Williams question when the continuum is singular}
In this section we prove Theorem \ref{main theorem1}.
In \cite{g-h} it is shown that if $\MQB$ is any $\aleph_1$-closed forcing notion \footnote{In fact being $\omega+1$-strategically closed is sufficient} of size $\leq 2^{\aleph_0}$ and if $\lambda$ is the least cardinal such that forcing with $\MQB$ adds a fresh $\lambda$-sequence of ordinals,
then forcing with $\MQB$ collapses $2^{\aleph_0}$ into $\lambda,$ and hence, if in addition $\lambda=\aleph_1,$ then $\MQB \simeq \Add(\aleph_1, 1).$
 Thus to prove Theorem 1.1, it  suffices to show
that if $\cf(2^{\aleph_0})=\aleph_1,$ then  any  $\aleph_1$-closed forcing notion $\MQB$ of size at most  $2^{\aleph_0}$ adds a fresh set of ordinals of size $\aleph_1.$
We give a  direct proof of this fact which is of its own interest, and avoids the use of the results of \cite{g-h}.

If $2^{\aleph_0}=\aleph_1,$ then the result is known to hold, so  assume that $\aleph_1 < 2^{\aleph_0}$ and $\cf(2^{\aleph_0})=\aleph_1.$ Let $\MQB$ be a non-trivial $\aleph_1$-closed forcing notion of size $\leq 2^{\aleph_0}$.
We are going to show that $\MQB$ is forcing equivalent to $\Add(\aleph_1, 1).$
\begin{notation}
For a forcing notion $\PP$ and a condition $p \in \PP,$ let $\PP \downarrow p$ denote the set of all conditions in $\PP$
which extend $p$; i.e., $\PP \downarrow p=\{ q \in \PP: q \leq_{\PP} p       \}$.
\end{notation}
Let $\langle \MQB_i: i < \omega_1       \rangle$ be a $\subseteq$-increasing and continuous
sequence of subsets of $\MQB$
such  that $\MQB_0=\emptyset$, for all $i<\omega_1, |\MQB_i| < 2^{\aleph_0}$, and $\MQB=\bigcup_{i<\omega_1}\MQB_i.$
\begin{lemma}
For every $i<\omega_1$ and every $p \in \MQB,$ there exists $q \leq_{\MQB} p$ such that there is no $r \in \MQB_i$ with
$r \leq q.$
\end{lemma}
\begin{proof}
Let $A$ be a maximal antichain in $\MQB $ below $p$ of size $2^{\aleph_0}$, which exists as $\MQB$ is non-trivial and $\aleph_1$-closed.
As $|\MQB_i|< 2^{\aleph_0},$ we can find $q \in A$ such that $(\MQB \downarrow q) \cap \MQB_i =\emptyset.$ Then $q$ is  as required.
\end{proof}
We now define by induction on $i<\omega_1$ a sequence $\bar{p}_i$ such that:
\begin{enumerate}
\item [(1)] $\bar{p}_i=\langle \bar{p}_i(\eta): \eta \in (^{i+1}$$(2^{\aleph_0}) )     \rangle$ is a maximal antichain in $\MQB,$

\item [(2)] If $j<i$ and $\eta \in (^{i+1}(2^{\aleph_0})),$ then $\bar{p}_i(\eta) \leq_{\MQB} \bar{p}_j(\eta \upharpoonright (j+1)),$

\item [(3)] If $\eta \in (^{i+1}(2^{\aleph_0}))$, then there is no member of $\MQB_i$ which is below $\bar{p}_i(\eta).$
\end{enumerate}
\begin{enumerate}
\item [$\underline{i=0}$]: Let $\bar{p}_0=\langle \bar{p}_0(\eta): \eta \in (^1(2^{\aleph_0}))      \rangle$  be any maximal antichain in $\MQB.$ Note that clauses $(2)$ and $(3)$ above are vacuous as $\MQB_0$ is empty.

\item [$\underline{i>0}$]:  For every $\nu \in (^{i}$$(2^{\aleph_0}))$ set $\bar{p}^1_{i,\nu}=\langle \bar{p}_j(\nu \upharpoonright (j+1)): j<i    \rangle$.
Then, by the induction hypothesis, $\bar{p}^1_{i,\nu}$ is a countable decreasing sequence of conditions in $\MQB$, and so the set
\[
\PP^2_{i,\nu}=\{ q \in \MQB: j<i \Rightarrow q \leq    \bar{p}_j(\nu \upharpoonright (j+1))   \}
\]
 is non-empty. Let
\[
\PP^3_{i,\nu}=\{ q \in \PP^2_{i,\nu}: \forall z \in \MQB_i [ z \nleq_{\MQB} q \text{~ and moreover ~} z \nVdash \text{``}q \in \dot{G}_{\MQB}\text{''}]         \}.
\]
$\PP^3_{i,\nu}$ is easily seen to be a dense subset of $\PP^2_{i,\nu},$ hence we can find a maximal antichain, say  $\bar{\PP}_{i, \nu}=\{ p_{i,\nu}(\alpha): \alpha < 2^{\aleph_0}  \}$, in it. For $\eta \in (^{i+1}$$(2^{\aleph_0}))$ set $\bar{p}_i(\eta)=p_{i, \eta \upharpoonright i}(\eta(i)).$ Then it is easily seen that  $\bar{p}_i=\langle \bar{p}_i(\eta): \eta \in (^{i+1}$$(2^{\aleph_0}))      \rangle$  is as required.
\end{enumerate}
Let $\lusim{v} \in V^{\MQB}$ be the $\MQB$-name
\[
\lusim{v} = \{ \langle(\check{j}, \check{\eta}(\check{j})),  \bar p_i(\eta) \rangle: j \leq i<\omega_1 \text{~and~}   \eta \in (^{i+1}(2^{\aleph_0}))          \}.
\]
\begin{claim}
$($a$)$ For every $i<\omega_1$ and $\eta \in (^{i+1}$$(2^{\aleph_0})),~ \bar p_i(\eta) \Vdash$``$\eta =\lusim{v} \upharpoonright i+1$''.

$($b$)$ $1_{\MQB}\Vdash$``$\lusim{v} \in (^{\omega_1}$$(2^{\aleph_0}))$''.
\end{claim}
\begin{proof}
$($a$)$ is clear from the definition of $\lusim v.$

$($b$)$ Let  $G$ be $\MQB$-generic over $V$. Then for each $i<\omega_1$ we can find a unique $\eta_i \in (^{i+1}(2^{\aleph_0}))$
such that $\bar p_i(\eta_i) \in G.$ If $j<i,$ then $\bar{p}_i(\eta_i) \leq_{\MQB} \bar{p}_j(\eta_i \upharpoonright (j+1)),$
and so $\eta_j = \eta_{i} \upharpoonright j+1.$
It then immediately follows that
\[
\lusim v[G] = \{(i, \eta_i(j)): j \leq i < \omega_1             \} = \bigcup_{i<\omega_1} \{  \eta \in (^{i+1}(2^{\aleph_0})): \bar p_i(\eta) \in G     \}
\]
is a function from $ \omega_1$ into $2^{\aleph_0}$.
\end{proof}
We now define a $\mathbb{Q}$-name $\lusim{\tau}$ for a function from $\omega_1$ into $2$  as follows: let $\langle \rho_\alpha: \alpha < 2^{\aleph_0}  \rangle$
be an enumeration of $^{\omega}2$ with no repetitions. Then let $\lusim \tau$ be such that
\[
\Vdash_{\MQB} \lusim{\tau}(\omega\cdot i+n)=\rho_{\lusim{v}(i)}(n).
\]
\begin{lemma}
$\Vdash_{\MQB}$``$\lusim{\tau} \in (^{\omega_1}2)$ and $\lusim{\tau} \notin \check{V}$''.
\end{lemma}
\begin{proof}
Let $q_1 \in \MQB.$ Then for some $i<\omega_1, q_1 \in \MQB_i.$ Since
 $\langle \bar p_i(\eta): \eta \in (^{i+1}$$(2^{\aleph_0}))     \rangle$ is a maximal antichain,  we can find $\eta \in (^{i+1}$$(2^{\aleph_0}))  $
such that $q_1$ is compatible with $ \bar p_i(\eta).$ But $q_1\nVdash$``$ \bar p_i(\eta) \in \dot{G}_{\MQB}$'',
so there is $q_2 \leq q_1$ such that $q_2$ is incompatible with $ \bar p_i(\eta).$ But again as $\langle  \bar p_i(\eta): \eta \in (^{i+1}$$(2^{\aleph_0}))      \rangle$ is a maximal antichain, there exists $\rho \in (^{i+1}$$(2^{\aleph_0}))$ such that $q_2$ and $ \bar p_i(\rho)$ are compatible. Let $q_3 \leq q_2,  \bar p_i(\rho),$
and let $j\leq i$ be maximal such that $\eta\upharpoonright j =\rho \upharpoonright j$ and $\eta\upharpoonright (j+1) \neq\rho \upharpoonright (j+1).$
Then $q_3,  \bar p_i(\eta)$ are compatible with $q_1$, but they force contradictory information about $\lusim{\tau} \upharpoonright [\omega\cdot j, \omega \cdot j+\omega).$ The result follows immediately.
\end{proof}
\begin{lemma}
\label{existence of suitable dense subset}
There exists a dense subset $\MQB'$ of $\MQB$ which is the union of $\aleph_1$-many maximal antichains $\langle I^*_i: i<\omega_1   \rangle$ of $\MQB.$
\end{lemma}
\begin{proof}
For any $p\in \MQB,$ by the previous lemma, $p$ does not force any value for $\lusim{\tau}$, hence there are  ordinal $i<\omega_1$ and conditions $p_0, p_1 \leq p$
such that $p_l\Vdash$``$\lusim{\tau}(i)=l$'', $l=0,1$. Hence we can define by recursion a sequence
\begin{center}
$\langle  \langle q_{p, \eta}, i_{p, \eta}, \sigma_{p, \eta} \rangle:  \eta \in~(^{<\omega}2)        \rangle$
\end{center}
such that:
\begin{enumerate}
\item [(4)] $q_{p, \langle \rangle} =p,$
\item [(5)] $\nu \lhd \eta \Rightarrow q_{p, \eta} \leq q_{p, \nu},$
\item [(6)] $i_{p, \eta}$ is the least ordinal $i$ less than $\omega_1$ such that $q_{p, \eta}$ does not decide $\lusim{\tau}(i),$
\item [(7)] $q_{p, \eta}\Vdash$``$\forall j< i_{p, \eta}, \lusim{\tau}(i_{p, \eta \upharpoonright j})=\sigma_{p, \eta}(j)$''.
\end{enumerate}
It is evident that if $\nu \lhd \eta,$ then $i_{p, \nu} < i_{p, \eta}.$
\begin{claim}
For any $p \in \MQB,$ there exists a perfect subtree $T_p$ of $^{\omega}2$ such that for some limit ordinal $\delta_p$
and every $\rho\in \Lim(T_p), \bigcup_n i_{p, \rho\upharpoonright n} =\delta_p$,
where $\Lim(T_p)$ is the set of all branches through $T_p$.
\end{claim}
\begin{proof}
For any $\eta \in (^{<\omega}2)$
set
\begin{center}
$\delta_{p, \eta}= \sup\{ i_{p, \nu}: \eta \lhd \nu \in(^{<\omega}2)      \}$.
\end{center}
For some $\eta_*,$ the ordinal $\delta_{p, \eta_*}$ is minimal.
$\delta_{p, \eta_*}$ is a limit ordinal of cofinality $\aleph_0,$ so let $\langle \eta_{p,m}: m<\omega \rangle$ be an increasing sequence cofinal in $\delta_{p, \eta_*}$ such that $\eta_{p,0}=\len(\eta_*).$ We define $h_m:$$~^{m}2 \rightarrow$$~^{\eta_{p,m}}2,$ by induction on $m<\omega,$ such that:
\begin{enumerate}
\item $h_m$ is $1$-$1$,

\item $h_0(\langle \rangle)=\eta_*,$

\item If $n<m$ and $\eta \in (^{m}2)$, then $h_n(\eta \upharpoonright n) \lhd h_m(\eta),$

\item If $\eta \in(^{m}2)$, then $i_{p, h_m(\eta)}> \eta_{p, m}.$

\end{enumerate}

Then $T_p=\{ h_m(\eta): m< \omega$ and $\eta \in(^{m}2)   \}$ and $\delta_p=\delta_{p, \eta_*}$ are as required.
\end{proof}
For each limit ordinal $\delta<\omega_1$ set
\[
I^1_\delta=\{ p\in \MQB: \delta_p=\delta \}.
\]
Then clearly $\MQB=\bigcup\{ I^1_\delta: \delta$ is a limit ordinal less than $\omega_1\}.$
\begin{claim}
Let $\delta$ be a  countable limit ordinal. Then there exists an antichain $\bar{q}^\delta=\langle q^\delta_{p}: p \in I^1_\delta          \rangle$
such that for each $p \in I^1_\delta, ~ q^\delta_{p} \leq p$.
\end{claim}
\begin{proof}
Let $\langle  p_\alpha: \alpha < \alpha_\delta \leq 2^{\aleph_0} \rangle$ enumerate $I^1_\delta.$ We choose, by induction on $\alpha,$ a pair
$\langle   r_\alpha, v_\alpha    \rangle$
such that
\begin{enumerate}
\item [(8)] $r_\alpha \leq p_\alpha,$ and $v_\alpha \in(^{\delta}2),$
\item [(9)] $r_\alpha\Vdash$``$\lusim{\tau} \upharpoonright \delta_{p_\alpha} = v_\alpha$'',
\item [(10)] $\alpha \neq \beta \Rightarrow v_\alpha \neq v_\beta.$
\end{enumerate}
Suppose $\alpha<\alpha_\delta$ and we have defined $\langle   r_\beta, v_\beta    \rangle$
for all $\beta< \alpha$ as above. We define $\langle   r_\alpha, v_\alpha    \rangle$.

For every $\rho \in \Lim(T_{p_\alpha}),$ the sequence $\langle  q_{p, \rho\upharpoonright n}: n<\omega      \rangle$
is a decreasing chain of conditions in $\MQB,$ and hence there is a condition $q^*_{\rho, \alpha}$ which extends all of them. We may further suppose that it forces a value  $v_{\rho, \alpha}$ for $\lusim{\tau} \upharpoonright \delta,$ where $\delta=\delta_{p_\alpha}.$ Also note that by the choice of $\langle  q_{\rho, \alpha}: \rho \in(^{<\omega}2)       \rangle$, for $\rho_1 \neq \rho_2$ in $\Lim(T_{p_\alpha}),$ we have
$v_{\rho_1, \alpha} \neq v_{\rho_2, \alpha}.$
Now $\{ v_\beta: \beta< \alpha  \} \subseteq $$~^{\delta}2$, hence for some  $\rho=\rho_\alpha \in \Lim(p_{\alpha})$
we have that $v_{\rho, \alpha} \notin \{ v_\beta: \beta< \alpha  \}.$ Let $r_\alpha=q_{\rho_\alpha, \alpha}$ and $v_\alpha=v_{\rho_\alpha, \alpha}.$
\end{proof}
Now for each limit ordinal $\delta<\omega_1$ let $I_\delta$ be a maximal antichain of $\MQB,$ such that $I_\delta \supseteq \{ q_{p, \delta}: p \in I^1_\delta   \},$ and let $\MQB'=\bigcup\{ I_\delta: \delta$ is a countable limit ordinal   $\}.$ Then
$\MQB'$ is as required which completes the proof of Lemma \ref{existence of suitable dense subset}.
\end{proof}
As each $I^*_i$ is a maximal antichain in $\MQB'$ and hence also in $\MQB,$
it can easily seen that there are $\bar{p}^*_i, i<\omega_1,$ such that
\begin{enumerate}
\item [(11)] $\bar{p}^*_i= \langle  p^*_\eta: \eta \in(^{i+1}(2^{\aleph_0}) )         \rangle$ is a maximal antichin of
$\MQB'$ (and hence of $\MQB$),

\item [(12)] If $j<i$ and $\eta \in(^{i+1}(2^{\aleph_0}))$, then  $p^*_\eta \leq p^*_{\eta \upharpoonright (j+1)},$

\item [(13)] If $i=j+1$ and $\eta \in(^{i+1}(2^{\aleph_0})),$ then $p^*_\eta$ is stronger  than some condition in $I^*_i.$
\end{enumerate}
Let
\[
\MQB''=\{ p^*_\eta: \exists i<\omega_1, \eta \in (^{i+1}(2^{\aleph_0}))         \}.
\]
\begin{lemma}
$\MQB''$ is a dense subset of $\MQB.$
\end{lemma}
\begin{proof}
Let $p\in \MQB.$ By Lemma \ref{existence of suitable dense subset}, we can find some $i<\omega_1$ and some $p_1\in I^*_i \subseteq \MQB'$ such that $p_1 \leq p.$
By $(13),$ each $p^*_\eta, \eta \in(^{i+1}(2^{\aleph_0})),$ is stronger  than some condition in $I^*_i.$
If there is no $\eta$ with $p^*_\eta \leq p_1,$ then we contradict with $(11).$ The result follows immediately.
\end{proof}
Finally note that the map
\[
\eta \mapsto p^*_\eta
\]
defines an isomorphism between a dense subset of $\Col(\aleph_1, 2^{\aleph_0})$ and $\MQB''.$ It follows that
\[
\MQB \simeq \MQB'' \simeq \Col(\aleph_1, 2^{\aleph_0}) \simeq \Add(\aleph_1, 1).
\]
The theorem follows. \hfill$\Box$
\section{A note on $\aleph_1$-closed forcing notions of size continuum}
In this section we present a result about $\aleph_1$-closed forcing notions of size continuum which will be used in
 section 5 for the proof of Theorem 1.3.

Assume that $2^{\aleph_0} =\aleph_2$ and that $\MRB$ is an   $\aleph_1$-closed forcing notions of size continuum which does not collapse
$\aleph_2.$ It then follows from \cite{g-h} that the forcing notion $\MRB$ does not add a fresh sequence of ordinals of size $\aleph_1$ and hence it is
$\aleph_2$-distributive. The following result is proved in \cite{balcar} Theorem 2.1.
\begin{lemma}
\label{base tree lemma}
There exists a sequence $\langle T_\alpha: \alpha< \aleph_2       \rangle$ of subsets of $\MRB$ such that:
\begin{enumerate}
\item Each $T_\alpha$ is a maximal antichain in $\MRB,$

\item  If $T=\bigcup\{ T_\alpha: \alpha < \aleph_2   \},$ then $(T, \geq_{\MRB})$ is a tree of height $\aleph_2,$ where $T_\alpha$
is the $\alpha$-th level of $T$,

\item  Each $t\in T$ has $\aleph_2$-many immediate successors,

\item $T$ is dense in $\MRB.$
\end{enumerate}
\end{lemma}
We denote the above tree $T$ by $T(\MRB)$, and call it a base tree of $\MRB.$ Note that by clause $(4)$, $\MRB \simeq T(\MRB).$

\section{Specializing  $\aleph_2$-trees which have few branches}
\label{Specializing aleph-2-trees which have few branches}
In this section we consider trees of size and height $\k$ which have $\leq \k$-many branches, and define a suitable forcing notion
 for specializing them. As we allow our trees to have cofinal branches,  we need a slightly different definition of the concept of a special tree than the usual ones.

 \begin{definition}
 Let $\kappa=\varrho^+,$ where $\varrho$ is a regular cardinal.
 \begin{enumerate}
 \item A  $\kappa$-tree is a tree of height and size $\kappa$ (so we allow the levels of the tree to have size $\kappa$).

 \item (~\cite{baumgartner}~)
 Let $T$ be a $\kappa$-tree. $T$ is  \emph{special} if there exists a function $F: T \rightarrow \varrho$ such that for all
 $x, y, z\in T$ if $x \leq_T y,z$ and $F(x)=F(y)=F(z),$ then either $y \leq_T z$ or $z \leq_T y.$
 \end{enumerate}
 \end{definition}
By \cite[Theorem 8.1]{baumgartner},  a $\kappa$-special tree
has at most $\kappa$-many cofinal branches.

Let  $\kappa=\varrho^+,$ where $\varrho$ is a regular cardinal.  Let also $\theta> \kappa$ be large enough regular and let  $\prec$ be a well-ordering of $H(\theta)$. Let $\Lambda$ denote the set of all $\kappa$-trees $T$ with at most
 $\kappa$-many cofinal branches, such that for all $t\in T, Suc_T(t),$ the set of successors of $t$ in $T$, has size $\kappa.$

We define a map $\star: \Lambda \to \Lambda$,
 where  to each $T \in \Lambda$, assigns a subtree $T^\star=\star(T)$ of $T$, such that $T^\star$ is dense in $T$ and it has no cofinal branches. Thus let $T \in \Lambda.$
Let $\langle b_\alpha: \alpha < \k  \rangle$ be the $\prec$-least  enumeration of the cofinal branches through $T$, and for each $\alpha < \k$ set
\begin{center}
 $s_\alpha =$ the $\leq_T$-least element of $b_\alpha \setminus \bigcup_{\beta<\alpha}b_\beta$.
\end{center}
Finally set
  \[
 T^\star=\{ t\in T: \neg(\exists \alpha, ~ s_\alpha <_T t \in b_\alpha )            \}.
 \]

 \begin{lemma}
 \label{t-star dense in T}
 $(a)$ $T^\star$ has no cofinal branches.

 $(b)$ $(T^\star, \geq_T)$ is dense in $(T, \geq_T)$ (when considered as forcing notions), in particular $(T^\star, \geq_T) \simeq (T, \geq_T)$.
\end{lemma}
\begin{proof}
$(a)$ Assume not, and let $b$ be a branch through $T^\star$. then for some $\alpha, b \subseteq b_\alpha,$ and then clearly $b \cap (T \setminus T^\star) \neq \emptyset,$
which is a contradiction.

$(b)$ Let $t\in T.$ If $t \in T^\star,$ then we are done; so assume that $t\notin T^\star.$ Then for some $\alpha < \k, s_\alpha <_T t \in b_\alpha.$ Let $t' \in Suc_T(t) \setminus \bigcup_{\beta \leq \alpha}b_\beta.$
Then $t' \in T^\star$ and $t' \geq_T t.$
\end{proof}
\begin{lemma}
\label{modifying specialization}
Assume there exists $F: T^\star \rightarrow \varrho$ such that if $F(x)=F(y),$ then $x$ and $y$ are incomparable in $T$. Then there exists
$F': T \rightarrow \varrho$ such that $F' \supseteq F$ and $F'$ specializes $T$.
\end{lemma}
\begin{proof}
Define $G: (T \setminus T^\star) \rightarrow \varrho$ as follows: Let $t\in (T \setminus T^\star)$.
Then for some $\alpha < \k$, $s_\alpha <_T t \in b_\alpha$. Set $G(t)=F(s_\alpha).$ It is now easily seen that $F'=F \cup G$ is as required.
\end{proof}
Thus, in order to define a forcing notion which specializes $T$, it suffices to define a forcing notion which adds a function $F: T^\star \rightarrow \varrho$ as in
Lemma \ref{modifying specialization}.

\begin{definition}
The forcing notion $\MQB(T^\star)$, for specializing $T^\star$,    is defined as follows:
\begin{enumerate}
\item [(a)]  A condition in $\MQB(T^\star)$ is a partial function $f: T^\star \rightarrow \varrho$ such that:
\begin{enumerate}
\item [(1)]$ \dom(f)$ has size $< \varrho$,
\item [(2)] If $x <_{T} y$ and $x,y\in \dom(f)$ then $f(x) \neq f(y)$.
\end{enumerate}
\item [(b)] $f \leq_{\MQB(T^\star)} g$ iff $f \supseteq g$.
\end{enumerate}
\end{definition}
It is clear that the forcing notion
$\MQB(T^\star)$ is  $\varrho$-directed closed. But in general, there is no guarantee that the forcing $\MQB(T^\star)$ satisfies the
$\k$-c.c., or  preserves all cardinals, even if we assume $\GCH$ (see \cite{cummings} and \cite{shelah-stanley}).
\begin{lemma}
\label{collapsing after specializing}
Forcing with $\MQB(T^\star) \ast T$ collapses $\k$ into $\varrho$.
\end{lemma}
\begin{proof}
By Lemma \ref{t-star dense in T}$(b)$, $\MQB(T^\star) \ast T \simeq \MQB(T^\star) \ast T^\star.$
Let $G$ be $\MQB(T^\star)$-generic over $V$ and  $H$ be $T^\star$-generic over $V[G].$  Let also
$F=\bigcup \{ f: f\in G\}$. Then $F: T^\star \rightarrow \varrho$ and  for all $x<_T y$ in $T^*$ we have
$F(x) \neq F(y).$
Let $b \in V[G \ast H]$ be a cofinal branch in $T^\star$.
 Then $F \upharpoonright b: b \rightarrow \varrho$
is an injection, and $|b|=\k$. Hence $\k$ is collapsed into $\varrho.$
\end{proof}
Given an infinite cardinal $\kappa,$ let $\Add(\aleph_0, \kappa)$
denote the Cohen forcing for adding $\kappa$-many new Cohen reals; thus conditions are finite partial functions
$p: \kappa \times \omega \to \{0, 1 \}$
ordered by reverse inclusion. The forcing is c.c.c., and hence it preserves all cardinals and cofinalities.

For our purpose in the next section, we will work with $\Add(\aleph_0, \k)$-names of trees, and we now modify the above
results to cover this case.

Assume $\kappa=\varrho^+,$ where $\varrho$ is a regular cardinal and let $\lusim{T}$ be an $\Add(\aleph_0, \k)$-name for a $\k$-tree which is forced to have $\leq \k$-many cofinal branches. Let $\lusim{T}^\star$ be an  $\Add(\aleph_0, \k)$-name such that it is forced by
 $\Add(\aleph_0, \k)$ that ``$\lusim{T}^\star$ is the subtree of $\lusim{T}$ defined using the function $\star$''.  We assume, without loss of generality, that it is forced by $\Add(\aleph_0, \k)$ that ``the set of nodes of $\lusim{T}^\star$ is $\k \times \k$ and for each $\alpha < \k,$ the $\alpha$-th level of $\lusim{T}^\star$ is $\{\alpha\} \times \k$''.
We now define $\MQB_A(\lusim{T}^\star)\in V$ as follows:

\begin{definition}
\begin{enumerate}
\item [(a)] A condition in $\MQB_A(\lusim{T}^\star)$ is a partial function  $f: \k \times \k \rightarrow \varrho$ such that:
\begin{enumerate}
\item [(1)] $\dom(f)$ is a  subset of $ \k \times \k$ of size $< \varrho$.
\item [(2)] If $x,y\in \dom(f)$  and $f(x) = f(y)$, then $\Vdash_{\Add(\aleph_0, \k)}$``$x$ and $y$ are incompatible in the tree ordering, $x \perp y$''. .
\end{enumerate}
\item [(b)] $f \leq_{\MQB_A(\lusim{T}^\star)} g$ iff $f \supseteq g$.
\end{enumerate}
\end{definition}
Note that we defined the forcing notion $\MQB_A(\lusim{T}^\star)$ in $V$ and not in the generic extension by $\Add(\aleph_0, \k).$
\begin{lemma}
\label{basic properties of specializing forcing for names}
\begin{enumerate}
\item [(a)] $\MQB_A(\lusim{T}^\star)$ is $\varrho$-directed closed.
\item [(b)] Let $G$ be  $\MQB_A(\lusim{T}^\star)$-generic over $V$. Then in $V[G]$, there exists a function $F: \k \times \k \rightarrow \varrho$,
such that for any $H$ which is $\Add(\aleph_0, \k)$-generic over $V[G]$, $F$ is a specializing function for $\lusim{T}^\star[H].$
\end{enumerate}
\end{lemma}

 The next lemma can be proved as in Lemma \ref{collapsing after specializing}.
\begin{lemma}
\label{collapsing cardinals with names}
Let $\lusim{T}$ be an $\Add(\aleph_0, \k)$-name for a $\k$-tree  which has $\leq \k$-many cofinal branches. Then
\begin{center}
$\Vdash_{\MQB_A(\lusim{T}^\star)\ast \lusim{\Add}(\aleph_0, \k)}$``~forcing with $\lusim{T}$ collapses $\k$''.
\end{center}
\end{lemma}
\begin{proof}
By Lemma \ref{t-star dense in T}$(b)$,
\[
(\MQB_A(\lusim{T}^\star)\ast \lusim{\Add}(\aleph_0, \k)) \ast \lusim{T} \simeq (\MQB_A(\lusim{T}^\star)\ast \lusim{\Add}(\aleph_0, \k)) \ast \lusim{T}^\star,
\]
and hence it suffices to show that
\begin{center}
$\Vdash_{\MQB_A(\lusim{T}^\star)\ast \lusim{\Add}(\aleph_0, \k)}$``~forcing with $\lusim{T}^\star$ collapses $\k$''.
\end{center}
Let $(G_1\ast G_2)*H$ be $(\MQB_A(\lusim{T}^\star)\ast \lusim{\Add}(\aleph_0, \k)) \ast \lusim{T}^\star$-generic over $V$ and
$F=\bigcup\{f: f\in G_1  \}.$ By Lemma \ref{basic properties of specializing forcing for names}, $F: \kappa \times \kappa \rightarrow \varrho$, and if
$T^\star=\lusim{T}^\star[G_2],$ then for $x<_{T^\star}y, F(x) \neq F(y)$.

Let $b \in V[(G_1\ast G_2)*H]$
be a cofinal branch of $T^\star$. Then $F \upharpoonright b: b \rightarrow \varrho$
is an injection, and hence $\k$ is collapsed into $\varrho.$
\end{proof}

\section{A negative answer to Williams question when the continuum is regular}
\label{A negative answer to Williams question when the continuum is regular}
In this section we prove Theorem \ref{main theorem2}.
In Subsection \ref{The main forcing construction and its basic properties} we define the main forcing construction $\MPB$ and prove some of its basic properties.
In Subsection \ref{preservation of cardinal kappa} we show that forcing with $\PP$ preserves $\kappa$.
Then in Subsection \ref{More on the forcing notion P} more properties of the forcing notion $\PP$
are proved and finally in subsection \ref{Completing the proof of Theorem 1.3} we complete the proof of Theorem \ref{main theorem2}.

\subsection{The main forcing construction and its basic properties}
\label{The main forcing construction and its basic properties}
In this subsection we define the main forcing notion that will be used in the proof of Theorem \ref{main theorem2}.
Let   $\kappa$ be a supercompact cardinal, and let  $\lambda > \kappa$ be a $2$-Mahlo cardinal. By \cite{laver}, We may  assume that $\kappa$ is Laver indestructible, i.e.,  the supercompactness of $\kappa$ is preserved under $\kappa$-directed
closed forcing notions, and that $\GCH$ holds at and above $\kappa.$

Let  $\Phi: \lambda \to H(\lambda)$ be such that for each $x \in H(\lambda), \Phi^{-1}(x) \cap \{\beta +2: \beta$ is Mahlo $\}$ is unbounded in $\lambda.$ Such a $\Phi$ exists as $|H(\lambda)|=2^{<\lambda}=\lambda$ and $\lambda$ is a $2$-Mahlo cardinal.

We define, by induction on $\a \leq \lambda,$  an iteration
\[
\PP=\langle \langle \PP_\alpha: \alpha \leq \lambda  \rangle, \langle \lusim{\MQB}_\alpha: \alpha < \lambda \rangle \rangle
\]
of forcing notions of length $\l.$
Suppose  $\a \leq \lambda$ and  we have defined $\PP_\beta,$ for all $\beta < \alpha$. We define
$\PP_\a$ as follows.
\begin{definition}
\label{main forcing definition for williams question}
 A condition $p$ is in $\PP_\a,$ if and only if $p$ is a function with domain $\alpha$ such that for every $\beta < \alpha, \Vdash_{\beta}$``$p(\beta) \in \lusim{\MQB}_\beta$'', where:
\begin{enumerate}
\item $\supp(p)$ has size less than $\k,$ where $\supp(p)$ denotes the the support of $p$.
\item $\{ \beta \in \supp(p):  \beta \equiv 0 (\text{mod~} 3)$ or $\beta \equiv 2 (\text{mod~} 3)        \}$
has size less than $ \aleph_1.$

\item If  $\beta< \kappa$ and  $\beta \equiv 0 (\text{mod~} 3)$ or $\beta \equiv 2 (\text{mod~} 3)$, then  $\Vdash_{\beta}$`` $\lusim{\MQB}_\beta=\lusim{\Col}(\aleph_1, \aleph_2+|\beta|)$''.

\item If $\beta \geq \kappa$, $\beta \equiv 0 (\text{mod~} 3)$ and $\beta$ is inaccessible, then $\Vdash_{\beta}$`` $\lusim{\MQB}_\beta=\lusim{\Add}(\aleph_1, \kappa)$''.

\item If $\beta \geq \kappa$, $\beta \equiv 1 (\text{mod~} 3)$ and $\beta$-$1$ is inaccessible, then $\Vdash_{\beta}$`` $\lusim{\MQB}_\beta=\lusim{\Col}(\kappa, 2^{|\PP_\beta|})=\lusim{\Col}(\kappa, 2^{|\beta|})$'' (as $|\PP_\beta|=|\beta|$).

\item If $\beta \geq \kappa$, $\beta \equiv 2 (\text{mod~} 3)$, $\beta$-$2$ is inaccessible and if $\Vdash_{\beta}$``$\kappa=\aleph_2$'' and $\Phi(\beta)$ is
a  $\PP_{\beta} * \lusim{\Add}(\aleph_0, \kappa)$-name for a $\kappa$-tree which has $\leq \kappa$-many
cofinal branches,
  then $\Vdash_{\beta}$`` $\lusim{\MQB}_\beta=\lusim{\MQB}_A(\Phi(\beta)^\star)$''.\footnote{$\Phi(\beta)^\star$ is defined from $\Phi(\beta)$ as in
  Section \ref{Specializing aleph-2-trees which have few branches} using a fixed well-ordering of a large initial segment of the universe.}

\item Otherwise, $\Vdash_{\beta}$`` $\lusim{\MQB}_\beta$
is the trivial forcing notion''.

\end{enumerate}
\end{definition}
Also set $\PP=\PP_{\lambda}$.

The next lemma gives  some basic properties of the forcing notion $\PP.$
\begin{lemma}
\label{basic properties of the forcing notion P}
\begin{enumerate}
\item [$(a)$] $\PP$ is $\aleph_1$-directed closed, and hence it preserves $CH$.

\item [$(b)$] If $\mu \in (\kappa, \lambda]$ is Mahlo,
then $\PP_\mu$ satisfies the $\mu$-c.c.

\item [$(c)$] $\PP_\lambda$ collapses all cardinals in $(\aleph_1, \kappa)$ into $\aleph_1,$ so if $\kappa$ is not collapsed, then $\Vdash_{\PP}$``$\kappa=\aleph_2$''.

\item [$(d)$] In $V^{\PP}, \lambda$ is preserved,
but all $\mu \in (\kappa, \lambda)$ are collapsed, so if $\kappa$ is not collapsed, then $\Vdash_{\PP}$``$\lambda=\kappa^+=\aleph_3$''.

\item [$(e)$] $\Vdash_{\PP}$``~$2^{\aleph_1}=\lambda$''.
\end{enumerate}
\end{lemma}
\begin{proof}
$(a)$ is clear as all forcing notions considered in the iteration are $\aleph_1$-directed closed and the support of the iteration is at least countable.

$(b)$ Assume $A \subseteq \PP_\mu$ is a maximal antichain of size $\mu$ and let $\langle p^\xi: \xi < \mu        \rangle$ be an enumeration
of $A$. Define $F: \mu \rightarrow \mu$ by
$F(\xi)=$ the least $\eta$ such that $\supp(p^\xi) \upharpoonright \xi \subseteq \eta.$ $F$ is a regressive function on the stationary set $X=\{\xi<\mu: \xi$ is inaccessible  $\}$, and hence $F$ is constant on some stationary subset $Y$ of $X$. Let $\eta$ be the resulting
fixed value. So for all $\xi \in Y, \supp(p^\xi) \upharpoonright \xi \subseteq \eta.$ We may further suppose that if $\xi_1 < \xi_2$
are in $Y$,
then $\supp(p^{\xi_1}) \subseteq \xi_2.$

As $\PP_\eta$ has size less than $\mu,$ there are $\xi_1 < \xi_2$ in $Y$ such that $p^{\xi_1} \upharpoonright \eta$ is compatible with
$p^{\xi_2} \upharpoonright \eta$. But then in fact $p^{\xi_1}$
is compatible with $p^{\xi_2}$ and we get a contradiction.

$(c)$, $(e)$ and the fact that forcing with $\PP_\lambda$ collapses all cardinals in $(\aleph_1, \kappa)$ into $\aleph_1$ are clear and the rest of $(d)$ follows from $(b).$
The lemma follows.
\end{proof}

\subsection{Preservation of $\k$}
\label{preservation of cardinal kappa}
In this subsection we show that forcing with $\PP$ preserves $\k$.
Let
\[
G=\langle \langle G_\alpha: \alpha \leq \lambda  \rangle, \langle H_\alpha: \alpha < \lambda \rangle \rangle
\]
be $\PP$-generic over $V$, i.e., each $G_\a$ is $\PP_\a$-generic over $V$ and $H_\a$ is $\lusim{\MQB}_\alpha[G_\a]$-generic over $V[G_\alpha].$

For each $\a \leq \l,$ we   define the  forcing notion
$\PP_\a^U \in V$, the $\PP_\a^U$-name $\lusim{\PP}_\a^C$ for a forcing notion, in such a way that:
\begin{enumerate}
\item [(a)] There are  projections $\chi_\a: \PP_\a \to \PP_\a^U$ and  $\pi_\a: \PP_\a^U \ast \lusim{\PP}_\a^C \to \PP_\a.$

\item [(b)] \label{transferring names}
If $G^U_\a=\chi_\a[G_\a]$, for
 $\a \leq \l$, then there exists a function $\Psi \in V[G^U_\lambda]$ such that for each ordinal $\alpha=\beta+2 > \kappa$, where $\beta$ is inaccessible,
$\Psi(\a) \in V[G^U_\a]$ is a $\PP^C_{\a} * \lusim{\Add}(\aleph_0, \kappa)$-name
 such that
$$\Phi(\a)[G^U_\a \ast H]=\Psi(\a)[H],$$
for any $H$ which is $\PP^C_{\a} * \lusim{\Add}(\aleph_0, \kappa)$-generic
over $V[G^U_\a].$ Further, if   $\Phi(\a)$ is
a  $\PP_{\a} * \lusim{\Add}(\aleph_0, \kappa)$-name for a $\kappa$-tree with $\leq \kappa$-many
cofinal branches, then in $V[G^U_\lambda], \Psi(\a)$  is a $\PP^C_{\a} * \lusim{\Add}(\aleph_0, \kappa)$-name for a $\kappa$-tree with $\leq \kappa$-many
cofinal branches.

\item [(c)]  $\PP_\a^U$ is $\kappa$-directed closed.

\item [(d)] For  every $\gamma \in [\a, \lambda]$, $\PP_\gamma^U\Vdash$`` $\lusim{\PP}_\a^C$ is $\aleph_1$-directed closed''. \footnote{It is also possible to define the forcing notions $\PP^U_\a \in V$  and  $\PP_\a^C \in V[G^U_\a]$ directly,  by setting
\[
\PP^U_\a = \{ p \in \PP_\a: \supp(p) \subseteq \{\b < \l: \beta \equiv 1 (\text{mod~} 3)\}  \}
\]
and
\[
\PP_\a^C=\{ p \in \PP_\a: \supp(p) \subseteq \{\b < \l: \beta \equiv 0 (\text{mod~} 3) \text{~or~} \beta \equiv 2 (\text{mod~} 3)   \} \text{~and~} p \text{~is compatible with ~} G^U_\a   \},
\]
where $G^U_\a=G_\a \cap \PP^U_\a.$ We will give a more explicit definition for $\PP_\a^C$, that will be useful in the proof of Lemma \ref{main lemma for forcing P-2}.
}

\end{enumerate}
 Let us first define the forcing notions $\PP^U_\a$ and the corresponding projections $\chi_\a$.
Let
\[
U=\{\b < \l: \beta \equiv 1 (\text{mod~} 3) \text{~and~} \b-1 \text{~is inaccessible~}    \}.
\]
For $\b \in U$, let
$\MQB^U_\b$ be the term forcing, whose conditions
are $\PP_\b$-names $\lusim p$ such that $\Vdash_{\PP_\b }$``$\lusim p \in \lusim{\Col}(\k, 2^{|\b|})$'', ordered by $\lusim p \leq_{\MQB^U_\b} \lusim q$ iff
$\Vdash_{\PP_\b }$``$\lusim p \leq_{\lusim{\Col}(\k, 2^{|\b|})} \lusim q$''.
Then set
$\PP_\a^U$ be the $< \kappa$-support product of the forcing notions $\MQB_\b^U,$ where $\b \in U \cap \alpha$. Then there is a natural projection
$\chi_\a: \PP_\a \to \PP_\a^U$.

We now define, by  induction on $\a \leq \lambda$, the $\PP_\a^U$-name $\lusim{\PP}_\a^C$  and the corresponding projection $\pi_\a: \PP_\a^U \ast \lusim{\PP}_\a^C \to \PP_\a.$ We also inductively verify (d) along the way.
\begin{enumerate}
\item \underline{$\a \leq \k$}: Let $\PP_\a^C$ to be $\PP_\a / \chi_\a^{-1}[\dot{G}^U_\a]$.  It is then clear that $\PP_\a^C \simeq \PP_\a$ and there exist a  forcing isomorphism  $\pi_\a: \PP_\a^U \ast \lusim{\PP}_\a^C \simeq \PP_\a.$ It is also clear that clause (d) holds.

\item \underline{$\a=\beta+1>\k,$ $\beta \equiv 0 (\text{mod~} 3)$ and $\beta$ is inaccessible}: Then set
\begin{center}
 $\Vdash_{\PP_\a^U}$``$\lusim{\PP}_\a^C \simeq \lusim{\PP}_\b^C \ast \lusim{\Add}(\aleph_1, \kappa)$''.
\end{center}
By the induction, there is a projection $\pi_\b: \PP_\b^U \ast \lusim{\PP}_\b^C \to \PP_\b.$ Since $\PP_\a^U=\PP_\b^U$, and since by clause (d) we can regard each $\PP^U_\b \ast \lusim{\PP}_\b^C$-name for an element of $\Add(\aleph_1, \kappa)$ as a $\PP_\b$-name,  this induces the projection $\pi_\a: \PP_\a^U \ast \lusim{\PP}_\a^C \to \PP_\a,$ which is defined by $\pi_\a(p, (\lusim q, \lusim r))=(\pi_\b(p, \lusim q), \lusim r).$ It is also clear that clause (d) holds in this case.

\item \underline{$\a=\beta+1>\k,$  $\beta \equiv 1 (\text{mod~} 3)$ and $\beta$-$1$ is inaccessible}: Let
$\Vdash_{\PP_\b^U}$``$\lusim{\PP}_\a^C=\lusim{\PP}_\b^C$''\footnote{Note that we can consider $\lusim{\PP}_\b^C$ as a $\PP_\a^U$-name as well.}.
In this case  $\PP_\a^U = \PP_\b^U \times \MQB^U_\b,$  and by the induction hypothesis, there is a projection $\pi_\b: \PP_\b^U \ast \lusim{\PP}_\b^C \to \PP_\b.$   This induces the projection $\pi_\a: \PP_\a^U \ast \lusim{\PP}_\a^C \to \PP_\a,$ which is defined by $\pi_\a((p, \lusim q), \lusim r)=(\pi_\b(p, \lusim r), \lusim q)$. Again clause (d) is easily verified.

\item \underline{$\a=\beta+1>\k,$  $\beta \geq \kappa$, $\beta \equiv 2 (\text{mod~} 3)$ and $\beta$-$2$ is inaccessible:}
Then we may assume that $\Phi(\beta)$ is
a  $\PP_{\beta} * \lusim{\Add}(\aleph_0, \kappa)$-name
for a $\kappa$-tree which has $\leq \kappa$-many
cofinal branches, as otherwise the forcing at stage $\beta$ is the trivial forcing and the result follows from the induction hypothesis.
By the induction hypothesis, we have  projections $\chi_\b,~\pi_\b,$
and so there exists $\Psi(\b) \in V[G^U_\b]$ as in clause (b) above. Set
\begin{center}
 $\Vdash_{\PP_\a^U}$``$\lusim{\PP}_\a^C \simeq \lusim{\PP}_\b^C \ast \lusim{\MQB}_A(\Psi(\beta)^\star)$''.
\end{center}
By the choice of $\Psi(\beta),$ it is clear that we  have a natural projection
$\pi_\a: \PP_\a^U \ast \lusim{\PP}_\a^C \to \PP_\a$, which extends $\pi_\b$. Clause (d) is easily verified in this case as well.

\item \underline{$\a=\b+1$ is not as above}: Then set
\begin{center}
 $\Vdash_{\PP_\a^U}$``$\lusim{\PP}_\a^C =\lusim{\PP}_\b^C$''.
\end{center}
Set also $\pi_\a=\pi_\b.$ It is also clear that clause (d) holds.

\item \underline{$\a$ is a limit ordinal}: Then set
\begin{center}
 $\Vdash_{\PP_\a^U}$``$\lusim{\PP}_\a^C$ is the countable support iteration of $\langle  \lusim{\PP}_\b^C : \b<\a       \rangle$''.
\end{center}
Note that  $\pi_\a$ can be defined in a uniform way from  $\pi_\b$'s, $\b < \a.$ To be more precise, let $\pi_\a:  \PP_\a^U \ast \lusim{\PP}_\a^C \to \PP_\a$
be defined by
\[
\pi_\a(p, \langle \lusim q_\b: \b< \a  \rangle)= \langle  \pi_\b(p, \lusim q_\b): \b < \a         \rangle.
\]
It is evident that $\pi_\a(1_{\PP_\a^U \ast \lusim{\PP}_\a^C})=1_{\PP_\a}$ and that $\pi_\a$ is order preserving. Now suppose $\langle p, \langle \lusim q_\b: \b< \a  \rangle \rangle \in \PP_\a^U \ast \lusim{\PP}_\a^C$, $r \in \PP_\a$ and suppose that $r \leq \pi_\a(p, \langle \lusim q_\b: \b< \a  \rangle)= \langle  \pi_\b(p, \lusim q_\b): \b < \a         \rangle.$ 
By \ref{main forcing definition for williams question}(2), the set
$$S=\{ \beta \in \supp(r):  \beta \equiv 0 (\text{mod~} 3) \text{~or~} \beta \equiv 2 (\text{mod~} 3)        \}$$
is at most countable. By induction on $\b \in S,$ we can find $(\bar p_\b, \lusim{\bar q}_\b) \in \PP_\b^U \ast \lusim{\PP}_\b^C$ such that
\begin{itemize}
\item [(6-1)] $\pi_\b(\bar p_\b, \lusim{\bar q}_\b) \leq r \restriction \b.$
\item [(6-2)] If $\b_0 < \b_1$ are in $S$, then $\bar p_{\b_1} \leq \bar p_{\b_0}$.
\end{itemize}
Each $\bar p_\b$, for $\b \in S$, is in $\PP^U_\a$, and since it is $\k$-closed, we can find  $\bar p \in \PP^U_\a$ which extends all $\bar p_\b$'s, $\b \in S$. Then $(\bar p, \langle \lusim{\bar q}_\b: \b< \a \rangle) \in \PP_\a^U \ast \lusim{\PP}_\a^C$, where for $\b \in \a \setminus S,$
$\Vdash_{\PP_\b^U}$``$\lusim{\bar q}_\b=1_{\lusim{\PP}_\b^C}$'', and it satisfies
\[
\pi_\a(\bar p, \langle \lusim{\bar q}_\b: \b \in S \rangle)= \langle \pi_\b(\bar p, \lusim{\bar q}_\b): \b<\a      \rangle \leq r.
\]

\end{enumerate}
Clause (d) follows easily from the induction hypothesis and the fact that $\PP^U_\b$'s are $\kappa$-directed closed and hence $\k$-distributive.

\begin{lemma}
\label{main lemma for forcing P-2}
For every $\a \leq \l,$  $\PP_\lambda^U\Vdash$`` $\lusim{\PP}_\a^C$ is  $\kappa$-c.c.''.
  \end{lemma}
  \begin{proof}
Let $G^U_\lambda$ be $\PP^U_\lambda$-generic over $V$ and for each $\beta < \lambda$ let
$G^U_\beta = G^U_\lambda \cap \PP^U_\beta$. Then $G^U_\beta$ is $\PP^U_\beta$-generic over $V$.
It follows  that for any ordinal  $\a=\beta+1>\k,$  where $\beta \equiv 2 (\text{mod~} 3)$ and $\beta$-$2$ is inaccessible, we have
\begin{center}
 $V[G^U_\lambda] \models$`` $\PP^C_{\alpha+1} \simeq \PP^C_{\a} \ast \lusim{\MQB}_A(\Psi(\a)^\star)$''.
 \end{center}
As $\PP^U_\l$ is $\k$-directed closed and $\k$ is assumed to be Laver indestructible,  $\kappa$ remains supercompact, and hence weakly compact, in  $V[G^U_\lambda]$.

 Working in $V[G^U_\lambda]$, let $\mathcal{F}$ be the weakly compact filter on $\kappa$, i.e., the filter on $\k$ ge\-ne\-rated by the sets $\{\l<\k: \, (V_\l, \in, B\cap V_\l)\models\psi\}$, where $B\subseteq V_\k$ and $\psi$ is a  $\Pi^1_1$ sentence for the structure $(V_\k, \in, B)$.
Let also $\mathcal{S}$ be the collection of $\mathcal{F}$-positive sets, i.e.,
$\mathcal{S}=\{ X \subseteq \k: \forall B \in \mathcal{F}, X \cap B \neq \emptyset          \}$.

The proof of the next claim is as in the proof of \cite[Lemma 2.13]{g-h2}, where the forcing notions $\PP^2_\a$ and $\PP^1_\a$ there, are replaced with $\PP^C_\a$
and $\Add(\aleph_0, \k)$.\footnote{In\cite{g-h2},  only forcing notions for specializing $\PP^1_\a$-names for trees are considered, while in our forcing,  we also consider the Cohen forcing $\Add(\aleph_1, \k)$, but this does not produce any problems, as the forcing $\Add(\aleph_1, \k)$ is well-behaved and is $\k$-c.c.}
 \begin{claim}
 \label{main corollary to g-h2 paper}
Work in $V[G^U_\lambda]$. Let  $\a \leq \lambda.$
For any sequence $\langle q_i: i<\kappa   \rangle$ of conditions in $\PP^C_{\a},$ there exist a set $X \in \mathcal{S}$  and two sequences
$\langle q^1_i: i \in X   \rangle$ and $\langle q^2_i: i \in X   \rangle$ of conditions in $\PP^C_{\a},$
 such that
 \begin{itemize}
 \item For all $i \in X, q^1_i, q^2_i \leq q_i$.
 \item For all $i<j$ in $X$, $q^1_i$ is compatible
with $q^2_j$, and this is witnessed by a condition $q$, such that for every
$\xi < \a, q \upharpoonright \xi \Vdash$``$q(\xi)=q^1_i(\xi) \cup q^2_j(\xi)$''.
\end{itemize}
 \end{claim}
By Claim \ref{main corollary to g-h2 paper}, for every $\a \leq \l,$
\begin{center}
$V[G^U_\l] \models$`` $\PP^C_\a$ is $\k$-c.c.'',
\end{center}
and the lemma follows.
\end{proof}
The following is immediate.
\begin{lemma}
\label{preservation of kappa}
$V^{\PP} \models \text{``~} \CH+ \kappa=\aleph_2+ \lambda=\aleph_3=2^{\aleph_1} \text{''}.$
\end{lemma}

\subsection{More on the forcing notion $\PP$}
\label{More on the forcing notion P}
In this subsection we prove a few more properties of the forcing notions $\PP$ that will be used in the proof of Theorem \ref{main theorem2}.
\begin{lemma}
\label{not adding new reals lemma}
Assume that  $\mu \in (\kappa, \lambda)$ is a Mahlo cardinal. Let $\lusim{T}$
be a $\PP_{\mu}*\lusim{\Add}(\aleph_0, \kappa)$-name of a $\k$-tree. Then
\[
\Vdash_{\PP_{\mu+2}*\lusim{\Add}(\aleph_0, \kappa)}\text{``~} \lusim{T} \text{~has~}\leq \kappa\text{-many~} \kappa\text{-branches}.
\]
\end{lemma}
\begin{proof}
Let  $G_\mu$ be $\PP_\mu$-generic over $V$ and $V_1=V[G_{\mu}]$. In $V_1$,  $\lusim{T}$ can be considered as an $\Add(\aleph_0, \kappa)$-name. Note that in $V_1$, $\kappa=\aleph_2,~ \mu=\aleph_3$ and $2^{\aleph_1}=\aleph_3.$ Further we have
\[
\PP_{\mu+2}/ \PP_{\mu} \simeq \MQB_{\mu}*\lusim{\MQB}_{\mu+1} = \Add(\aleph_1, \kappa)* \lusim{\Col}(\kappa, 2^\mu),
\]
and
\[
\Vdash_{\PP_{\mu+2}*\lusim{\Add}(\aleph_0, \kappa)}\text{``~} |\{ b \in V_1: b  \text{~is a branch of~} \lusim{T}       \}| \leq |(2^\kappa)^{V_1}|=\k \text{''}.
\]
So it suffices to show that forcing with
  $\MQB_{\mu}*\lusim{\MQB}_{\mu+1}*\lusim{\Add}(\aleph_0, \kappa)$ adds no new cofinal branches. Assume by contradiction that  $\lusim{\eta}$ is a $\MQB_{\mu}*\lusim{\MQB}_{\mu+1}*\lusim{\Add}(\aleph_0, \kappa)$-name which is forced to be a new $\kappa$-branch of $\lusim{T}$.
The next claim follows easily from the assumption that $\lusim{\eta}$ is forced to be a new branch.
\begin{claim}
\label{what force new branch}
For every
 $ \langle p^0, p^1, p^2 \rangle \in \MQB_{\mu}*\lusim{\MQB}_{\mu+1}*\lusim{\Add}(\aleph_0, \kappa)$, there are  conditions
 $\langle q_i^0, q_i^1, q^2_i\rangle$, for $i=0,1,$ $\delta<\kappa$ and $x_0, x_1$ such that:
\begin{enumerate}
\item [(a)] $\langle q_0^0, q_0^1, q^2_0\rangle, \langle q_1^0, q_1^1, q^2_1\rangle \leq \langle p^0, p^1, p^2 \rangle$,

\item [(b)] $x_0 \neq x_1,$

\item [(c)] $\Vdash$``$x_0, x_1 \in \lusim{T}_\delta$, the $\delta$-th level of $\lusim{T}$'',

\item [(d)] $\langle q_i^0, q_i^1, q_i^2\rangle\Vdash$``$x_i \in \lusim{\eta}$'' (i=0,1). \hfill$\Box$
\end{enumerate}
\end{claim}
In fact, as the forcing notions $\Add(\aleph_1, \kappa)$ and $\Add(\aleph_0, \kappa)$ are $\kappa$-c.c. and $\Col(\kappa, 2^\mu)$ is forced to be $\kappa$-closed, we can show that the conditions
$\langle q_0^0, q_0^1, q^2_0\rangle$ and  $\langle q_1^0, q_1^1, q^2_1\rangle$ in the claim can be chosen so that $q_0^0=q_1^0=p^0$
and $q_0^2=q^2_1=p^2$ (see \cite{komjath} for similar arguments).

Let us assume that the empty condition forces $\lusim{\eta}$ is a new branch. By repeated application of Claim \ref{what force new branch}, we can build
 a sequence  $\langle  \lusim{q}^1_\nu: \nu \in (^{< \omega_1}2)        \rangle$
of $\MQB_\mu$-names of elements of $\lusim{\MQB}_{\mu+1}$,
 an increasing continuous sequence $\langle \delta_i: i < \omega_1    \rangle$ of ordinals less than $\kappa$ and a sequence
 $\langle  x_\nu:    \nu \in (^{< \omega_1}2)        \rangle$ such that:
 \begin{enumerate}
 \item $\nu_1 \unlhd \nu_2 \Rightarrow \Vdash_{\MQB_\mu}$``$\lusim{q}^1_{\nu_2} \leq \lusim{q}^1_{\nu_1}$'',

  \item $\langle \emptyset, \lusim{q}^1_\nu, \emptyset \rangle \Vdash$``$x_{\nu} \in \lusim{T}_{\delta_i}$'' where $i=\len(\nu)$,

 \item $x_{\nu^{\frown} \langle 0 \rangle} \neq x_{\nu^{\frown} \langle 1 \rangle},$

\item $\langle \emptyset, \lusim{q}^1_\nu, \emptyset \rangle \Vdash$``$x_\nu \in \lusim{\eta}$'',

\item $\nu_1 \unlhd \nu_2 \Rightarrow \langle \emptyset, \lusim{q}^1_{\nu_2}, \emptyset \rangle \Vdash$``$x_{\nu_1} <_{\lusim{T}} x_{\nu_2}$''.
\end{enumerate}
For some $\xi < \kappa$, $\langle  \lusim{q}^1_\nu: \nu \in (^{<\omega_1}2)        \rangle$ is in fact an $\Add(\aleph_1, \xi)$-name. Now we have
\[
\MQB_{\mu}*\lusim{\MQB}_{\mu+1}*\lusim{\Add}(\aleph_0, \kappa) \simeq \Add(\aleph_1, \xi)*\lusim{\Add}(\aleph_1, [\xi, \kappa))*\lusim{\MQB}_{\mu+1}*\lusim{\Add}(\aleph_0, \kappa),
\]
and in the generic extension $V^{\PP_\mu * \lusim{\Add}(\aleph_1, \xi)},$ we have an interpretation $q^1_\nu$ of the name $\lusim{q}^1_\nu$,
 where $\nu \in (^{<\omega_1}2).$

Work in $V^{\PP_\mu * \lusim{\MQB}_\mu}.$
For each $\tau \in (^{\omega_1}2),$ let $q^1_\tau \leq q^1_{\tau \upharpoonright i}, i<\omega_1$
and let $\delta=\sup\{ \delta_i: i<\omega_1  \} < \kappa.$ By extending $q^1_\tau$ if necessary, we can assume that for some $x_\tau,$
\[
\langle \emptyset, q^1_\tau, \emptyset \rangle \Vdash\text{``}~x_\tau \in \lusim{T}_\delta \cap \lusim{\eta}\text{~''}.
\]
But then for all $\tau_1 \neq \tau_2$ in $^{\omega_1}2$ we have $x_{\tau_1} \neq x_{\tau_2},$
and so
\begin{center}
$\Vdash_{\PP_\mu * \lusim{\Add(\aleph_0, \kappa)}}$``~ the $\delta$-th level of the tree  has at least $2^{\aleph_1}=\mu=\aleph_3$-many nodes''.
\end{center}
 But
$\Vdash_{\PP_\mu * \lusim{\Add}(\aleph_0, \kappa)}$``$~|\lusim{T}_\delta| \leq \kappa < \mu$'',
and
we get a contradiction.
\end{proof}
The next lemma follows from Lemma \ref{not adding new reals lemma} and the fact that $\PP_\mu * \lusim{\Add}(\aleph_0, \kappa) \lessdot \PP_{\mu+2} * \lusim{\Add}(\aleph_0, \kappa)$
\begin{lemma}
With the same hypotheses as in Lemma \ref{not adding new reals lemma}, we have the following:
In $V^{\PP_{\mu+2}}$, $\lusim{T}$ is isomorphic to some $\lusim{T}',$ which is an $\Add(\aleph_0, \kappa)$-name of a $\k$-tree
 with $\leq \kappa$-many cofinal branches.

\end{lemma}
\subsection{Completing the proof of Theorem \ref{main theorem2}}
\label{Completing the proof of Theorem 1.3}
Finally in this subsection we complete the proof of Theorem \ref{main theorem2}.  We first prove the following lemma.
\begin{lemma}
\label{complementing lemma}
 $\Vdash_{\PP*\lusim{\Add}(\aleph_0, \kappa)} \text{``~ Any~}\aleph_1\text{-closed forcing notion of size~} \leq \kappa \text{~collapses~} \kappa$''.
\end{lemma}
\begin{proof}
Let $G \ast H$ be $\PP \ast \Add(\aleph_0, \kappa)$-generic over $V$ and assume $\MRB \in V[G\ast H]$
is an $\aleph_1$-closed forcing notion of size $\leq \kappa=\aleph_2$.

 Assume towards a contradiction that  forcing with $\MRB$ over $V[G\ast H]$ does not collapse $\aleph_2$. It then follows from \cite{g-h2} that $\MRB$ is $\aleph_2$-distributive, and hence by Lemma \ref{base tree lemma}, there exists a $\k$-tree $T=T(\MRB)$, the base tree of $\MRB,$  which is dense in $\MRB.$

Let $\lusim{T}$ be a $\PP \ast \Add(\aleph_0, \kappa)$-name for $T$.
 By Lemma \ref{basic properties of the forcing notion P}, we may assume that $\lusim{T} \in H(\lambda)$, and hence there
  exists some Mahlo cardinal $\beta \in (\kappa, \lambda)$  such that $\lusim{T}$ is a $\PP_{\beta} * \lusim{\Add}(\aleph_0, \kappa)$-name.
  By Lemma \ref{not adding new reals lemma},  $\lusim{T}$ is isomorphic to some $\lusim{T}' \in H(\lambda)$ which is a $\PP_{\beta+2} * \lusim{\Add}(\aleph_0, \kappa)$-name for a
  $\k$-tree which has $\leq \kappa$-many cofinal branches. On the other hand
   $\{\beta+2: \beta \in (\k, \l)$ is a Mahlo cardinal and   $\Phi(\beta+2)=\lusim{T}'   \}$
in unbounded in $\l$, and hence we can choose $\beta$ as above  such that
    $\lusim{T}'=\Phi(\beta+2).$
Then $\PP_{\beta+3} \simeq \PP_{\beta+2} \ast \lusim{\MQB}_A(\Phi(\beta+2)^\star)$, and by Lemma \ref{collapsing cardinals with names},
\begin{center}
$\Vdash_{\PP_{\beta+3}*\lusim{\Add}(\aleph_0, \kappa)}$``~Forcing with $\lusim{T'}$ collapses $\kappa$ into $\aleph_1$''.
\end{center}
As $\PP_{\beta+3} \ast \lusim{\Add}(\aleph_0, \kappa) \lessdot \PP \ast \lusim{\Add}(\aleph_0, \kappa),$
\begin{center}
$\Vdash_{\PP \ast \lusim{\Add}(\aleph_0, \kappa)}$``~Forcing with $\lusim{T'}$ collapses $\kappa$ into $\aleph_1$''.
\end{center}
This implies
\begin{center}
$\Vdash_{\PP * \lusim{\Add}(\aleph_0, \kappa)}$``~Forcing with $\lusim{\MRB}$ collapses $\kappa$ into $\aleph_1$'',
\end{center}
The lemma follows.
\end{proof}
Now let
  $G \ast H$ be $\PP \ast \Add(\aleph_0, \kappa)$-generic over $V$. Let also $\MRB \in V[G\ast H]$
be an $\aleph_1$-closed forcing notion of size $\leq \kappa=\aleph_2$.
 By Lemma \ref{complementing lemma},
 \begin{center}
$V[G \ast H] \models$`` $\MRB \simeq \Col(\aleph_1, \kappa) \simeq \Add(\aleph_1, 1)$''.
\end{center}

\section{Consistency, every forcing which adds a fresh subset of $\aleph_2$  collapses a cardinal}
\label{Every forcing which adds a new subset of aleph-2 can collapse a cardinal}
In this section we prove Theorem \ref{main theorem3}. Thus assume that $\GCH$ holds and $\lambda > \kappa$ are such that $\kappa$ is supercompact and Laver indestructible, and $\lambda$ is a $2$-Mahlo cardinal. Let also $\Phi: \lambda \rightarrow H(\lambda)$ be such that for each $x \in H(\lambda), \Phi^{-1}(x) \cap \{\beta+2: \beta$ is Mahlo $ \}$ is unbounded in $\lambda$.
The forcing notion we define is very similar the forcing notion of Section \ref{A negative answer to Williams question when the continuum is regular}.
\begin{definition}
Let
\[
\langle \langle \PP_\alpha: \alpha \leq \lambda  \rangle, \langle \lusim{\MQB}_\alpha: \alpha < \lambda \rangle \rangle
\]
be an iteration of forcing notions such that for each $\a \leq \l, p \in \PP_\a$ if and only if $p$ is a function with domain $\a$
such that:
\begin{enumerate}
\item $p$ has support of size less than $ \kappa$.
\item
$\{ \beta \in \supp(p):  \beta \equiv 0 (\text{mod~} 3)$ or $\beta \equiv 2 (\text{mod~} 3)        \}$
has size less than $ \aleph_1,$

\item If  $\beta< \kappa$ and  $\beta \equiv 0 (\text{mod~} 3)$ or $\beta \equiv 2 (\text{mod~} 3)$, then  $\Vdash_{\beta}$`` $\lusim{\MQB}_\beta=\lusim{\Col}(\aleph_1, \aleph_2+|\beta|)$'',

\item If $\beta \geq \kappa$, $\beta \equiv 0 (\text{mod~} 3)$ and $\beta$ is inaccessible, then $\Vdash_{\beta}$`` $\lusim{\MQB}_\beta=\lusim{\Add}(\aleph_1, \kappa)$'',

\item If $\beta \geq \kappa$, $\beta \equiv 1 (\text{mod~} 3)$ and $\beta$-$1$ is inaccessible, then $\Vdash_{\beta}$`` $\lusim{\MQB}_\beta=\lusim{\Col}(\kappa, 2^{|\PP_\beta|})$,

\item If $\beta \geq \kappa$, $\beta \equiv 2 (\text{mod~} 3)$  $\beta$-$2$ is inaccessible, and $\Phi(\beta)$ is a $\PP_{\beta}$-name for $\k$-tree with $\leq \kappa$-many
cofinal branches, then $\Vdash_{\beta}$`` $\lusim{\MQB}_\beta=\lusim{\MQB}_A(\Phi(\beta)^\star)$'',

\item Otherwise, $\Vdash_{\beta}$`` $\lusim{\MQB}_\beta$
is the trivial forcing notion'',
\end{enumerate}
Finally set $\PP=\PP_{\lambda}.$
\end{definition}
The next lemma can be proved as in Section \ref{A negative answer to Williams question when the continuum is regular}.
\begin{lemma}
Let $G$ be $\PP$-generic over $V$. Then the following hold in $V[G]$:

$(a)$ $2^{\aleph_0}=\aleph_1 < \kappa=\aleph_2 < 2^{\aleph_1}=\lambda=\aleph_3,$

$(b)$ Every Tree of size and height $\aleph_2$ is specialized.
\end{lemma}
Thus $(a)$-$(c)$ of Theorem \ref{main theorem3} are satisfied. Let's prove Theorem \ref{main theorem3}.$(d)$.
The proof is  similar to Todorcevic's proof in \cite{todorcevic}, and we present it here for completeness.

Work in $V[G].$
Let $\PP$ be any forcing notion, and suppose that forcing with $\PP$  adds a fresh subset of $\aleph_2$ without collapsing it. We show that forcing with
$\PP$ collapses $\aleph_3.$ Let $\mathbb{B}=RO(\PP)$. Let also $\lusim{\tau}$ be a name for a fresh subset of $\aleph_2$ such that
\[
\| (\lusim{\tau} \subseteq \aleph_2) \wedge   (\lusim{\tau} \notin \check{V}) \wedge (\forall \alpha < \aleph_2, \lusim{\tau} \cap \alpha \in \check{V})   \|_{\mathbb{B}}=1
\]
For $\alpha<\aleph_2,$ set $a_{\alpha,0}=\| \alpha \in \lusim{\tau}  \|_{\mathbb{B}}$ and  $a_{\alpha,1}=\| \alpha \notin \lusim{\tau}  \|_{\mathbb{B}}.$
Let $T_0=\{1_{\mathbb{B}}\}$, and for $0<\alpha < \aleph_2$ set
\[
T_\alpha = \{ \bigwedge\{ a_{\beta, f(\beta)}: \beta<\alpha  \}: f \in\text{}^{\alpha}\text{}2, \bigwedge\{ a_{\beta, f(\beta)}: \beta<\alpha  \} \neq 0_{\mathbb{B}}   \}.
\]
By the assumption on $\lusim{\tau},$ each $T_\alpha$ is a partition of $1_{\mathbb{B}},$ for $\beta < \alpha, T_\alpha$ refines $T_\beta$
and so
$T= \bigcup \{T_\alpha: \alpha < \aleph_2  \}$ is a tree of height $\aleph_2,$ whose $\alpha$-th level is $T_\alpha.$ Also clearly
$|T|=2^{\aleph_1}=\aleph_3.$
\begin{claim}
\label{bounding names below aleph-2}
For every $0_{\mathbb{B}} \neq b\in \mathbb{B},$ there exists $\alpha < \aleph_2$ such that
\[
|\{a \in T_\alpha: a \wedge b \neq 0_{\mathbb{B}}  \}| > \aleph_2.
\]
\end{claim}
\begin{proof}
Suppose not. So we can find $0_{\mathbb{B}} \neq b\in \mathbb{B}$ such that for each $\alpha < \aleph_2,$
$|\{a \in T_\alpha: a \wedge b \neq 0_{\mathbb{B}}  \}| \leq \aleph_2$. Define a new tree $T^*= \bigcup \{T^*_\alpha: \alpha < \aleph_2  \},$
where for each $\alpha,$
\[
T^*_\alpha = \{a \wedge b: a \in T_\alpha, a \wedge b > 0_{\mathbb{B}}  \}.
\]
Then $T^*$ is an $\aleph_2$-tree of size $\aleph_2$ and hence it is specialized. But then
\[
\| \dot{G}_{\mathbb{B}} \cap T^* \text{~ is a new cofinal branch of~} T^*               \|_{\mathbb{B}} \geq b,
\]
where $\dot{G}_{\mathbb{B}}$ is the canonical name for a generic ultrafilter over $\mathbb{B}.$ This is impossible as $T^*$ is specialized and
forcing with $\mathbb{B}$ preserves $\aleph_2.$
\end{proof}
For each $\alpha < \aleph_2,$ let $\langle  a_\alpha(\xi): \xi < \lambda_\alpha \leq \aleph_3          \rangle$
be an enumeration of $T_\alpha,$ and let $\lusim{f}$ be a name for a function from $\aleph_2$ into $\aleph_3$ defined by
\begin{center}
 $\|\lusim{f}(\alpha)=\xi \|_{\mathbb{B}}=$ $\left\{
\begin{array}{l}
         a_\alpha(\xi)  \hspace{1.6cm} \text{ if } \xi < \lambda_\alpha,\\
         0_{\mathbb{B}} \hspace{2.2cm} \text{Otherwise}.
     \end{array} \right.$
\end{center}

\begin{claim}
$\|\text{range}(\lusim{f})$ is unbounded in $\aleph_3        \|_{\mathbb{B}}=1_{\mathbb{B}}$.
\end{claim}
\begin{proof}
Assume not. Then for some $\delta <\aleph_3,$ $b=\|\text{range}(\lusim{f}) \subseteq \delta     \|_{\mathbb{B}}>0_{\mathbb{B}}$.
By Claim \ref{bounding names below aleph-2}, we can find $\alpha < \aleph_2$ such that $|\{a \in T_\alpha: a \wedge b \neq 0_{\mathbb{B}}  \}| = \aleph_3,$ so $\lambda_\alpha=\aleph_3.$
Pick some $\xi> \delta$ so that $a_{\alpha}(\xi) \wedge b \neq 0_{\mathbb{B}}$. This implies
\begin{center}
$\|\lusim{f}(\alpha)=\xi \|_{\mathbb{B}} \wedge \|\text{range}(\lusim{f}) \subseteq \delta     \|_{\mathbb{B}} \neq 0_{\mathbb{B}}$,
\end{center}
and we get a contradiction (as $\xi > \delta$).
 \end{proof}

School of Mathematics, Institute for Research in Fundamental Sciences (IPM), P.O. Box:
19395-5746, Tehran-Iran.

E-mail address: golshani.m@gmail.com

Einstein Institute of Mathematics, The Hebrew University of Jerusalem, Jerusalem,
91904, Israel, and Department of Mathematics, Rutgers University, New Brunswick, NJ
08854, USA.

E-mail address: shelah@math.huji.ac.il

\end{document}